\documentclass[12pt]{amsart}  
 
\usepackage{amsmath,amssymb,amsfonts}
\usepackage[hidelinks]{hyperref}
\numberwithin{equation}{section}

\providecommand{\k}{}\renewcommand{\k}{\mathbf{k}}
\providecommand{\C}{}\renewcommand{\C}{\mathcal{C}}
\providecommand{\D}{}\renewcommand{\D}{\mathcal{D}}
\providecommand{\M}{}\renewcommand{\M}{\mathcal{M}}
\providecommand{\Vec}{}\renewcommand{\Vec}{\mathrm{Vec}}
\providecommand{\sVec}{}\renewcommand{\sVec}{\mathrm{sVec}}
\providecommand{\Ver}{}\renewcommand{\Ver}{\mathrm{Ver}}
\providecommand{\Cl}{}\renewcommand{\Cl}{\mathrm{Cl}}
\providecommand{\h}{}\renewcommand{\h}{\mathfrak{h}}

\providecommand{\id}{}\renewcommand{\id}{\operatorname{id}}
\providecommand{\Hom}{}\renewcommand{\Hom}{\operatorname{Hom}}
\providecommand{\End}{}\renewcommand{\End}{\operatorname{End}}
\providecommand{\Pic}{}\renewcommand{\Pic}{\operatorname{Pic}}
\providecommand{\Br}{}\renewcommand{\Br}{\operatorname{Br}}

\theoremstyle{plain}
\newtheorem{theorem}{Theorem}[section]
\newtheorem{lemma}[theorem]{Lemma}
\newtheorem{proposition}[theorem]{Proposition}
\newtheorem{corollary}[theorem]{Corollary}

\theoremstyle{definition}

\newtheorem{example}[theorem]{Example}

\theoremstyle{remark}
\newtheorem{remark}[theorem]{Remark}

\begin{document}

\title{Clifford and Weyl algebras in symmetric tensor categories}

\author{Pavel Etingof}
\address{Department of Mathematics,
Massachusetts Institute of Technology,
77 Massachusetts Avenue,
Cambridge, MA 02139,
USA
}
\email{etingof@math.mit.edu}

\begin{abstract} Let $\C$ be a symmetric tensor category over an algebraically closed field $\k$ of characteristic $\ne 2$. We study Clifford and Weyl algebras of objects of $\C$ with a (skew-)symmetric bilinear form. When the form is non-degenerate, we establish simplicity and the Azumaya property for such algebras under suitable assumptions. We also compute Clifford and Weyl algebras in the Verlinde category $\Ver_p$ and use them to prove that if $\C$ is Frobenius exact then the Weyl algebra of a symplectic object of $\C$ with finite symmetric algebra is Azumaya. Using this, we introduce the symplectic Witt group $\mathcal S\mathcal W(\C)$, the subgroup of the Brauer group $\Br(\C)$ consisting of Morita classes of such Azumaya algebras, and describe it for $\C={\rm Rep}(G)\boxtimes\sVec$ by Clifford-group extensions when $|G|$ is invertible in $\k$, and by second Stiefel--Whitney classes when $\k=\mathbb C$. Without Frobenius exactness, the Weyl algebra of a symplectic object need not be simple; its simplicity is characterized by a degreewise contraction criterion. In odd characteristic we construct a canonical Azumaya quotient for the Weyl algebra of every symplectic object in a symmetric tensor category of moderate growth.
\end{abstract}

\maketitle

\tableofcontents

\section{Introduction} 

Clifford and Weyl algebras, which are quantizations of exterior and symmetric algebras respectively, are among the simplest and most important examples of non-commutative algebras. In this paper we develop a theory of Clifford and Weyl algebras in an arbitrary symmetric tensor category $\C$ and establish their basic properties (such as the correspondence between Clifford and Weyl algebras by parity change and the PBW theorem). We also establish the simplicity of these algebras under suitable assumptions; in particular, we show that Weyl algebras of objects with finite symmetric algebras in Frobenius exact categories are Azumaya algebras. Finally, we define the symplectic Witt group $\mathcal S\mathcal W(\C)$ of $\C$ - the subgroup of the Brauer group $\Br(\C)$ consisting of Morita classes of such Azumaya algebras, 
 and compute it for the category of super-representations of a finite group $G$. Over an algebraically closed field in which $|G|$ is invertible, Subsection~\ref{sec:cliffgroup} uses pullbacks of the Clifford-group extension. Over $\k=\mathbb C$, Subsection~\ref{sec:swgroup} identifies these pullbacks with the images of classical second Stiefel--Whitney classes. 

The organization of the paper is as follows. In Section 2 we discuss the theory of Clifford and Weyl algebras of arbitrary (skew)-symmetric forms. In Section 3 we discuss Clifford and Weyl algebras of non-degenerate forms, define symplectic Witt groups, and compute the symplectic Witt group 
of the category of super-representations of a finite group. Section~\ref{sec:reduced} treats arbitrary symmetric tensor categories. We give a criterion for the Weyl algebra of a symplectic object to be simple, exhibit the failure of simplicity in $\Ver_{p^n}$ for $n\ge2$, and construct a canonical Azumaya quotient of such Weyl algebras using the corresponding Nichols algebras.

{\bf Acknowledgements.} In this paper I collaborated with ChatGPT 5.5 and 5.6 Pro. ChatGPT 6.2 Astra assisted with the present revision and the proofs in Section~\ref{sec:reduced}. This work was partially supported by the NSF grant DMS-2502467. I am grateful to V. Ostrik and to the referee for useful comments.

\section{Clifford and Weyl algebras} 

\subsection{Definition and basic properties of Clifford and Weyl algebras} 
Throughout the paper, $\k$ will denote an algebraically closed field of characteristic $\ne 2$. Let $\C$ be a symmetric tensor category over $\k$ (\cite{EGNO}, Definitions 4.1.1, 8.1.1, 8.1.12).   
Let $V$ be an object of $\C$ equipped with a bilinear form $B\in \Hom(V\otimes V,\mathbf{1})$. Suppose $B$ is either symmetric, i.e., $B\circ c_{V,V}=B$, or skew-symmetric, i.e., $B\circ c_{V,V}=-B$, where $c$ is the symmetric braiding of $\C$. In this case $\h_V:=V\oplus \mathbf{1}$ has a natural structure of a Lie superalgebra in $\C$, i.e., 
the image of a Lie algebra in $\C\boxtimes \sVec$ under the (non-symmetric) forgetful functor to $\C$. Namely,
if $B$ is skew-symmetric then $\h_V$ is purely even and is called the {\it Heisenberg algebra} of $V$; 
the commutator maps $V\otimes V$ to $\mathbf{1}$ by $B$, and $\mathbf{1}$ is central. If $B$ is symmetric, 
the definition is the same, except $\mathbf{1}$ is even but $V$ is odd; in this case we say that 
$\h_V$ is the {\it Heisenberg superalgebra} of $V$. 

Define the associative (ind-)algebra $A(V)$ in $\C$ 
to be the quotient of the enveloping algebra $U(\h_V)$ by the relation identifying the unit 
copy of $\mathbf{1}$ in $U(\h_V)$ with $\mathbf{1}\subset \h_V$. In other words, if $z_U:\mathbf 1\to\mathfrak h_V\to U(\mathfrak h_V)$ is the central generator and $1_U:\mathbf 1\to U(\mathfrak h_V)$ is the algebra unit, then 
$$
A(V):=U(\mathfrak h_V)/({\rm Im}(z_U-1_U)).
$$
We call 
$A(V)$ the {\it Weyl algebra} of $V$ if $B$ is skew-symmetric and 
the {\it Clifford algebra} of $V$ if $B$ is symmetric. For $B=0$ we specify the symmetry type as part of the data. The skew-symmetric choice gives $SV$, and the symmetric choice gives $\wedge V$. 
For $\C=\Vec$, this agrees with the definitions 
in classical algebra.\footnote{Note that many classical references normalize the Clifford algebra relation as $vw+wv=2B(v,w)$. In this paper we drop the factor of $2$.} 

The relations of $A(V)$ are invariant under the automorphism $-{\rm Id}$ on $V$, hence 
$A(V)=A_+(V)\oplus A_-(V)$, where this automorphism acts on $A_\pm(V)$ by $\pm 1$. 
In the Clifford algebra case, it is natural to think of this $\mathbb Z/2$-grading as a super-grading. 
Note also that changing the sign of $B$ transforms $A(V)$ into $A(V)^{\rm op}$, where in the Clifford case ${\rm op}$ 
denotes the super-opposite algebra. Thus for every choice of $i=\sqrt{-1}\in \k$ we obtain an isomorphism 
$A(V)\cong A(V)^{\rm op}$ (as rescaling each argument by $i$ transforms $B$ to $-B$). 

Let $\psi\in \C$ be a super-line, i.e., an invertible object of dimension $-1$ with an identification $\alpha: \psi\otimes \psi\cong \mathbf{1}$ (so $c_{\psi,\psi}=-\id_{\psi\otimes\psi}$); for example,
there is a natural one if $\C=\C_+\boxtimes \sVec$, where $\C_+$ is another tensor category and $\psi$ is the generating simple object of $\sVec$. We will identify $\psi\otimes \psi$ with ${\mathbf 1}$ using $\alpha$. 
If an object $V$ is equipped with a bilinear form $B$, then so is $V\otimes \psi$, but this 
functor exchanges symmetric and skew-symmetric forms. 

Let 
$$
A_\psi(V):=A_+(V)\oplus A_-(V)\otimes \psi
$$
with the natural multiplication.
Explicitly, the product on $A_\psi(V)$ restricts to the original product on $A_+(V)$ and the $A_+(V)$-bimodule structure on $A_-(V)$, while the product of two odd terms is
\[
(A_-(V)\otimes\psi)\otimes(A_-(V)\otimes\psi)\xrightarrow{\;(m\otimes\alpha)(\id\otimes c_{\psi,A_-(V)}\otimes \id)\;}A_+(V).
\]

\begin{lemma}\label{leweyl1} (i) There is a natural isomorphism 
$
A(V\otimes \psi)\cong A_\psi(V).
$

(ii) Let $V,W\in \C$ be two objects equipped with bilinear forms of the same type, i.e., both symmetric or both skew-symmetric. Then 
$$
A(V)\otimes A(W)\cong A(V\oplus W),
$$ 
where in the case of symmetric forms this is the super-tensor product.
\end{lemma} 

\begin{proof}
For (i), define the form on $V\otimes\psi$ by
\[
(V\otimes\psi)^{\otimes2}\xrightarrow{\;\id_V\otimes c_{\psi,V}\otimes\id_\psi\;}V^{\otimes2}\otimes\psi^{\otimes2}
\xrightarrow{\;B\otimes\alpha\;}\mathbf{1}.
\]
Because $c_{\psi,\psi}=-\id$, this changes the symmetry type of $B$. The generating object $V\otimes\psi$ is sent to the odd part $A_-(V)\otimes\psi$ of $A_\psi(V)$, and the defining commutation relation is exactly as above. Hence the universal property of the enveloping-algebra quotient gives $A(V\otimes\psi)\cong A_\psi(V)$.

 (ii) In the Weyl case, the generating objects $V$ and $W$ commute because the cross-pairing is zero, so the universal property gives 
$$
A(V)\otimes A(W)\cong A(V\oplus W).
$$ 
In the Clifford case the same argument is made in superalgebras: the cross-pairing is zero, so the odd generating objects from the two summands super-commute, which is precisely the super-tensor product convention.
\end{proof} 

Lemma \ref{leweyl1}(i) shows that in symmetric tensor categories containing a super-line, the theories of Weyl algebras and Clifford algebras are equivalent. 

\subsection{Contraction maps} 

Recall that for any $W\in \C$, the symmetric algebra $SW$ is a Hopf algebra with coproduct 
decomposing into homogeneous components $\Delta_{i,j}: S^{i+j}W\to S^iW\otimes S^jW$. 
In particular, by duality the components $\Delta_{1,j-1}$ define maps 
$$
D_j: W^*\otimes S^jW\to S^{j-1}W
$$ 
which combine into the contraction map 
$$
D: W^*\otimes SW\to SW.
$$ 
Thus if $V\in \C$ is equipped with a bilinear form $B$ then we have a map 
$$
D_{B,ij}: S^iV\otimes S^jV\to S^{i-1}V\otimes S^{j-1}V
$$
given by
\[
 D_{B,ij}=(\id\otimes B\otimes\id)
       (\Delta_{i-1,1}\otimes\Delta_{1,j-1}).
\]
Here $i,j\ge1$, and we take the bicontraction to be zero if either degree is zero.
These maps 
combine into the bicontraction map
$$
D_B: SV\otimes SV\to SV\otimes SV.
$$  

More generally, given $k\in \mathbb Z_{\ge 0}$, subsets $S\subset [1,i],T\subset [1,j]$ with 
$|S|=|T|=k$, and a bijection $\sigma: S\to T$, we have 
a morphism 
$$
B^{S,T,\sigma}_{ijk}: V^{\otimes i}\otimes V^{\otimes j}\to S^{i-k}V\otimes S^{j-k}V
$$
obtained by contracting the $S$-tensorands of $V^{\otimes i}$ with $T$-tensorands 
of $V^{\otimes j}$ according to $\sigma$ using $B$ and then projecting to symmetric powers. 
Then the morphism 
$$
B_{ijk}:=\sum_{S,T,\sigma} B^{S,T,\sigma}_{ijk}
$$
is $\mathfrak S_i\times \mathfrak S_j$-invariant, hence descends to a morphism 
$$
D_{B,ij}^{(k)}: S^iV\otimes S^jV\to S^{i-k}V\otimes S^{j-k}V.
$$ 
These morphisms combine into the map 
$$
D_B^{(k)}: SV\otimes SV\to SV\otimes SV.
$$ 
It is easy to see that 
$$
D_B^k=k! D_B^{(k)},
$$
in particular if ${\rm char}(\k)=p>0$ then $D_B^p=0$. 

\subsection{The PBW theorem}
The algebra $A(V)$ has an increasing $\mathbb Z_{\ge 0}$-filtration $F^\bullet$ defined by $\deg(V)=1$, and we have a natural surjective homomorphism of graded algebras
$\xi: SV\to {\rm gr}A(V)$ for Weyl algebras and $\xi: \wedge V\to {\rm gr}A(V)$ for Clifford algebras. 

\begin{proposition}\label{pbw} (PBW theorem) $\xi$ is an isomorphism. Thus $A(V)$ is a filtered quantization of the Poisson 
algebra $SV$ or super-Poisson algebra $\wedge V$. 
\end{proposition} 

\begin{proof} By replacing $\C$ with $\C\boxtimes \sVec$ if necessary, we may assume without loss of generality that $\C$ contains a super-line. Then by Lemma \ref{leweyl1}(i), it suffices to prove the statement for Weyl algebras.
In this case, there is a natural morphism
$$
\beta_k: SV\otimes SV\to SV
$$ 
of degree $-2k$ given by 
$$
\beta_k:=m_0\circ D_B^{(k)},
$$
where $m_0$ is the multiplication on $SV$. For example, $\beta_0$ is the usual product and $\beta_1$ is the Poisson bracket defined by $B$. 

We can now define the {\it Moyal--Weyl} product on $SV$ by the formula 
$$
m:=\sum_{k\ge 0}2^{-k}\beta_k=m_0\circ E,\ E:=\sum_{k\ge 0}2^{-k}D_B^{(k)}.
$$
(Note that in characteristic zero $E$ can be expressed by the usual Moyal--Weyl formula
$E=\exp(\tfrac{1}{2}D_B)$).
It is easy to see that 
$$
E\circ (m_0\otimes \id)=(m_0\otimes \id) \circ E_{13}E_{23}
$$
as maps $SV\otimes SV\otimes SV\to SV\otimes SV$, where $E_{rs}$ denotes $E$ applied to the $r$-th and $s$-th factors. Indeed, a matching from the first two factors to the third splits uniquely according to the source of each pair, and its weight $2^{-k}$ is the product of the weights of the two matchings. 

It follows that the product $m$ is associative. Indeed, 
$$
m\circ (m\otimes \id)=m_0\circ E\circ (m_0\otimes \id)\circ (E\otimes \id)=
m_0\circ (m_0\otimes \id) \circ E_{13}E_{23}E_{12}
$$
and likewise 
$$
m\circ (\id\otimes m)=m_0\circ (\id \otimes m_0) \circ E_{12}E_{13}E_{23}.
$$
The two orders of contraction enumerate the same matchings of disjoint tensor positions in the three factors. Each matching gives the same morphism in either order by symmetry. This proves associativity, including in degrees at least $p$ in characteristic $p$.

It follows that there is a natural homomorphism $\phi: A(V)\to (SV,m)$ sending $V$ to $V$ by the identity map, and we have $\operatorname{gr}(\phi)\circ \xi=\id_{SV}$. Thus $\xi$ is an isomorphism. The filtrations start in degree zero and are exhaustive, so induction on the filtration also gives the filtered isomorphism $\phi:A(V)\xrightarrow{\sim}(SV,m)$. 
\end{proof} 

\begin{remark} We give a second proof of Proposition~\ref{pbw} for skew-symmetric forms. Given an object $V\in \C$ with a skew-symmetric form $B$, define the affine group scheme $\mathcal H_V$ in $\C$, which 
as a scheme is the object $V\oplus {\mathbf 1}$ (i.e., $\mathcal O(\mathcal H_V)=S(V^*\oplus {\mathbf 1})$) with group law given by the following schematic formula (i.e., formula at the level of points over commutative algebras in $\C$): 
$$
(v_1,\alpha_1)*(v_2,\alpha_2):=(v_1+v_2,\alpha_1+\alpha_2+\tfrac{1}{2}B(v_1,v_2)).
$$
We call this group the {\it Heisenberg group} of $V$. Differentiating the group commutator gives $[v,w]=B(v,w)z$, where $z$ generates the central summand. Thus $\operatorname{Lie}(\mathcal H_V)=\h_V$ (cf.~\cite[Example~4.8(2)]{E}). To prove PBW for the quotient in which $z$ acts as $1$, we realize $\h_V$ by operators on $SV$. Let $M_v$ denote multiplication by $v$ on $SV$ and $D_v$ contraction with $B(v,-)$. In schematic notation put $T_v=M_v+\tfrac12D_v$. With commutators taken using the symmetry of $\C$, the $M_v$ commute, the $D_v$ commute, and $[D_v,M_w]=B(v,w)\id$. Skew-symmetry gives $[T_v,T_w]=B(v,w)\id$, so the $T_v$ define an action of $A(V)$ on $SV$. Evaluation at $1$ gives a filtered map $a\mapsto a\cdot1$ from $A(V)$ to $SV$. The degree-$d$ term of $T_{v_1}\cdots T_{v_d}(1)$ is $v_1\cdots v_d$, since each contraction lowers degree. Thus the associated graded map is a left inverse to $\xi$, proving the proposition. 
\end{remark} 

\section{Clifford and Weyl algebras of non-degenerate forms} 
From now on assume that $B$ is non-degenerate, i.e., defines an isomorphism $V\to V^*$ by evaluating the first (or, equivalently, second) factor.\footnote{Indeed, the two morphisms
$B_1,B_2:V\longrightarrow V^*$ obtained by evaluating the first and the second tensorand of $B$ respectively are transposes of one another under rigidity (with the symmetry inserted). Hence $B_1$ is an isomorphism if and only if $B_2$ is an isomorphism.}

\subsection{Simplicity of Clifford and Weyl algebras} 

\begin{proposition}\label{simplic} Let $R$ be a simple (ind-)algebra in $\C$.

(i) Suppose $V$ is equipped with a non-degenerate skew-symmetric bilinear form $B$, and $S^p V=0$ (hence $S^iV=0$ for $i\ge p$) if $p={\rm char}(\k)>0$. Then the algebra $A_R(V):=R\otimes A(V)$ is simple.

(ii) Suppose $V=W\oplus W^*$ with the skew-symmetric form defined by the pairing between $W$ and $W^*$, and 
$S^pW=S^pW^*=0$ (hence $S^iW=S^iW^*=0$ for $i\ge p$) if $p={\rm char}(\k)>0$. Then the algebra $A_R(V):=R\otimes A(V)$ is simple.
\end{proposition}

\begin{example} Let $\C=\sVec$, $V\in \C$ a purely odd space (i.e., $V=V_0\otimes \psi$, where $V_0$ is a usual vector space over $\k$), and $R=\k$. For a non-degenerate symmetric form $B_0$ on $V_0$, the Clifford algebra $\Cl(V_0,B_0)$ is simple \emph{as a superalgebra}. Proposition~\ref{simplic}(i) applies directly in characteristic zero, and in characteristic $p>2$ when \mbox{$\dim V_0<p$}, since $S^pV=\wedge^p V_0\otimes\psi^{\otimes p}$. In arbitrary dimension, decompose $(V_0,B_0)$ into orthogonal lines. Each corresponding odd line has zero symmetric square, so Corollary~\ref{simplic1} applies in characteristic $p$. Note that if $\dim V_0$ is odd, 
it is not simple as an ordinary algebra; rather, it is the direct sum of two matrix blocks. 
\end{example} 

\begin{proof}[Proof of Proposition~\ref{simplic}] The filtration on $A(V)$ extends naturally to $A_R(V)$ by setting $\deg R=0$. 

(i) Let $i>0$, and identify $V$ with $V^*$ by evaluating $B$ on the first argument. Recall that $D_i: V\otimes S^iV\to S^{i-1}V$ corresponds by duality to the map $\Delta_{1,i-1}: S^iV\to V\otimes S^{i-1}V$.
Composing $\Delta_{1,i-1}$ with multiplication $V\otimes S^{i-1}V\to S^iV$ gives multiplication by $i$. Hence in characteristic zero or for $i<p$ in characteristic $p$ the map $\Delta_{1,i-1}$ is a split injection. 

It follows that for such $i$, if $E\subset R\otimes S^iV$ is a nonzero object, then $(1_R\otimes D_i)(V\otimes E)\ne 0$.
Thus if $E\subset F^iA_R(V)$ is a nonzero subobject whose image in
$\operatorname{gr}^iA_R(V)=F^iA_R(V)/F^{i-1}A_R(V)$
is nonzero, then the leading term of $[\,,\,](V\otimes E)$ is the contraction above, so the subobject $[\,,\,](V\otimes E)\subset F^{i-1}A_R(V)$ is nonzero. Now let $I$ be a nonzero two-sided ideal in $A_R(V)$ and suppose $i$ is the smallest integer such that $I\cap F^iA_R(V)\ne 0$. Denote this intersection by $E$. If $i>0$ then as we have shown, $[\,,\,](V\otimes E)\subset I\cap F^{i-1}A_R(V)$ is nonzero, a contradiction with minimality of $i$. Thus $i=0$, so $I\cap R\ne 0$. Since $R$ is simple, $I\cap R=R$, so $I$ contains the unit of $A_R(V)$, hence $I=A_R(V)$. 

(ii) The summands $W$ and $W^*$ are isotropic. By PBW, ordered multiplication identifies $A_R(W\oplus W^*)$, as a filtered object, with $R\otimes SW\otimes S(W^*)$, equipped with the total-degree filtration. Its component of bidegree $(a,b)$ is $R\otimes S^aW\otimes S^b(W^*)$. On the associated graded, commutation with $W^*$ contracts the $W$-factors, and commutation with $W$ contracts the $W^*$-factors, up to sign.

Fix a positive total degree $d$. On a nonzero component with $a+b=d$, choose the first contraction if $a>0$ and the second if $a=0$. The adjoint of the chosen contraction is monic: on the relevant symmetric power it is $\Delta_{1,r-1}$ for $r=a$ or $r=b$, and multiplication after this map is $r\,\id$. In positive characteristic every nonzero component has $a,b<p$. For each of the two contractions, different source bidegrees land in different target bidegrees. Thus the two adjoint maps on total degree $d$ have zero common kernel. As in (i), a nonzero ideal meeting the filtration for the first time in positive degree would have a nonzero commutator in smaller degree. It must therefore meet $R$, and simplicity of $R$ gives the result.
\end{proof} 

\begin{corollary}\label{simplic1} Let $\operatorname{char}(\k)=p>0$ and 
$$
V=V_1\oplus\cdots\oplus V_n,\ B=B_1\oplus\cdots\oplus B_n,
$$ 
where the $B_i$ are non-degenerate skew-symmetric bilinear forms on $V_i$. If for all $1\le m\le n$ either $S^pV_m=0$ or $V_m=W_m\oplus W_m^*$ with $B_m$ defined by the pairing of $W_m,W_m^*$ with $S^pW_m=S^pW_m^*=0$ then $A(V)$ is a simple algebra. 
\end{corollary} 

\begin{proof} The proof is by induction in $n$ starting with the trivial case $n=0$. Take $n>0$. Let $V'=V_1\oplus\cdots\oplus V_{n-1}$. Then by Lemma \ref{leweyl1}(ii), $A(V)=A(V')\otimes A(V_n)$. But $A(V')$ is simple by the induction assumption, so $A(V)$ is simple by Proposition \ref{simplic}. 
\end{proof} 

We choose the hyperbolic form on $W\oplus W^*$ to be
$B_{\rm hyp}((w,f),(u,g))=f(u)-g(w)$ in schematic notation.
Then $A(W\oplus W^*)$ acts naturally on $SW$, with $W$ acting by multiplication and 
$W^*$ by contraction. When $SW$ has finite length this gives a homomorphism 
$$
\eta: A(W\oplus W^*)\to \underline{\End}(SW).
$$ 

Also recall from \cite{OZ} that an {\it Azumaya algebra} in $\C$ is a nonzero algebra $A\in\C$ such that
the natural map $A\otimes A^{\rm op}\to \underline{\End}(A)$ is an isomorphism. Morita equivalence classes of 
Azumaya algebras form a group called the {\it Brauer group} of $\C$, denoted $\Br(\C)$. By \cite{DN}, Subsection 5.5, 
if $\C$ is finite then there is a natural isomorphism $\Br(\C)\to \Pic(\C)$ which sends $A$ to $A$-mod. 

\begin{corollary}\label{simplic2} Let $\C$ be a symmetric tensor category, assumed Frobenius exact in positive characteristic (cf.~\cite{CEO}).

(i) If $V\in \C$ is equipped with a non-degenerate skew-symmetric form and the symmetric algebra $SV$ has finite length then $A(V)$ is a simple algebra.

(ii) If $W\in \C$ and the symmetric algebra $SW$ has finite length then $\eta$ is an isomorphism.

(iii) In (i) $A(V)$ is a (central) Azumaya algebra. 
\end{corollary}

Recall that a symmetric tensor category $\D$ is said to have {\it moderate growth} if for every 
$X\in \D$ the length $\ell(X^{\otimes n})$ grows at most exponentially with $n$. 
To apply a fiber functor in the proof of Corollary~\ref{simplic2}, we first check moderate growth. 

\begin{lemma}\label{lemo} Let $\D$ be a symmetric tensor category over $\k$ tensor generated by an object $V$ (i.e., 
every object of $\D$ is a subquotient of a finite direct sum of tensor words in $V$ and $V^*$). If the algebra 
$SV$ or $\wedge V$ has finite length then $\D$ has moderate growth. 
\end{lemma} 

\begin{proof} We may assume that $SV$ has finite length: otherwise, if $\wedge V$ has finite length, 
replace $\D$ with $\D\boxtimes \sVec$ and $V$ by $\psi\oplus V\otimes \psi$ (the category tensor generated by this object is $\D\boxtimes \sVec$). Thus $S^NV=0$ for some $N$. So in the case ${\rm char}(\k)=0$ the category $\D$ 
is Schur-finite, hence has moderate growth (\cite{De}); in fact, by the main result of 
\cite{De} it is super-Tannakian. 

So we may assume that ${\rm char}(\k)=p>0$. 
Then the symmetrizing element
\[
\mathbf e_N:=\sum_{\sigma\in\mathfrak S_N}\sigma\in\k\mathfrak S_N
\]
factors through $V^{\otimes N}\twoheadrightarrow S^NV$, and hence acts by zero on $V^{\otimes N}$. Since $\mathbf e_N\ne0$ in $\k\mathfrak S_N$, this gives $\operatorname{sd}(V)<\infty$ in the notation of \cite[Definition~4.1]{CEO}.
Since duality intertwines the symmetric-group actions on $V^{\otimes n}$ and $(V^*)^{\otimes n}$ (up to the involution $\sigma\mapsto\sigma^{-1}$ of $\k\mathfrak S_n$), it follows that $\operatorname{sd}(V^*)=\operatorname{sd}(V)<\infty$. Hence \cite[Proposition~4.7(6)]{CEO} gives $\operatorname{ad}(V),\operatorname{ad}(V^*)<\infty$. Lemma~4.9(2) of \cite{CEO} then implies that
$Z:=\mathbf1\oplus V\oplus V^*$ has finite alternating degree, and hence finite growth degree by \cite[Proposition~4.7(5)]{CEO}. For any fixed $X\in\D$, there are $r,M\ge1$ such that $X$ is a subquotient of $(Z^{\otimes r})^{\oplus M}$. Therefore $X^{\otimes n}$ is a subquotient of $(Z^{\otimes rn})^{\oplus M^n}$, and
\[
 \ell(X^{\otimes n})\le M^n\ell(Z^{\otimes rn}),
\]
which grows at most exponentially. Thus $\D$ has moderate growth.
\end{proof} 

\begin{proof}[Proof of Corollary~\ref{simplic2}] Without loss of generality, we may assume that $\C$ is tensor generated by $V$ in (i),(iii) and by $W$ in (ii): otherwise replace $\C$ by the full tensor subcategory generated by $V$ or $W$. Then by Lemma \ref{lemo}, $\C$ has moderate growth. Hence if ${\rm char}(\k)=0$ then $\C$ is super-Tannakian by \cite{De}, thus all parts of the corollary follow from classical algebra. So we may assume that ${\rm char}(\k)=p>0$. 
In this case, by \cite{CEO}, $\C$ admits a fiber functor $F: \C\to \Ver_p$.

(i) Since $F$ is exact and faithful, it suffices to show that $F(A(V))=A(F(V))$ is simple. Now, $F(V)=\oplus_{i=1}^{p-1}M_i\otimes L_i$, where $M_i$ are finite-dimensional $\k$-vector spaces and $L_i$ are simple objects of $\Ver_p$. Since $SV$ has finite length, so does $SF(V)=F(SV)$, so $M_1=0$.
But for $i>1$, $S^pL_i=0$ (\cite[Section~6]{EOV}). Also, $L_i$ is self-dual, and its invariant bilinear form is skew-symmetric for $i$ even and symmetric for $i$ odd. Thus a skew form on the isotypic component $M_i\otimes L_i$ corresponds to a symmetric form on $M_i$ if $i$ is even, and to a skew form on $M_i$ if $i$ is odd. Hence the form $B$ decomposes into a direct sum of forms on $V_i=L_i$ with $i$ even and forms on $V_j=L_j\oplus L_j$ for odd $j$ pairing the factors. More explicitly, for $i$ even the multiplicity space $M_i$ carries a non-degenerate symmetric form, so it is an orthogonal sum of lines, giving copies of $L_i$. For $i$ odd the multiplicity space $M_i$ carries a non-degenerate skew-symmetric form, so it is a sum of symplectic planes, giving summands $L_i\oplus L_i$. Thus the simplicity of $A(V)$ follows from Corollary \ref{simplic1}. 

(ii) It suffices to prove the statement after applying $F$, i.e., in $\Ver_p$. We have $F(W)=\oplus_{i=1}^{p-1}N_i\otimes L_i$, and since $SW$ has finite length, $N_1=0$. Thus by Lemma \ref{leweyl1}(ii) it suffices to prove the statement for 
$W=L_i$, $i\ge 2$.

For $W=L_i$ with $i\ge2$, both $W$ and $W^*$ have zero $p$-th symmetric power. Proposition~\ref{simplic}(ii) therefore makes the unital action $\eta$ injective. For $0\le r<p$, the averaging idempotent in $\k\mathfrak S_r$ identifies symmetric coinvariants with invariants. Dualizing gives
\[
 S^r(W^*)\cong(S^rW)^* .
\]
All higher symmetric powers of $W$ and $W^*$ vanish. By ordered multiplication the source of $\eta$ is thus isomorphic, as an object, to $SW\otimes(SW)^*$. The monomorphism $\eta$ is an isomorphism by finite length.

(iii) Write $b(v)=B(v,-)$. The opposite algebra of $A(V,B)$ is $A(V,-B)$. On the Moyal--Weyl realization $SV$, the left action of $u$ is $M_u+\tfrac12D_{b(u)}$ and the right action of $v$, taken with the symmetry, is $M_v-\tfrac12D_{b(v)}$. These formulas follow from the product with a factor of degree one. The isometry to the hyperbolic space is
\[
 (V,B)\oplus(V,-B)\longrightarrow V\oplus V^*,\qquad
 (u,v)\longmapsto \bigl(u+v,\tfrac12 b(u-v)\bigr).
\]
Under this isometry the left-right action is the action by multiplication and contraction from (ii). Hence $A(V)\otimes A(V)^{\rm op}\to\underline{\End}(A(V))$ is an isomorphism.
\end{proof}

\subsection{The top symmetric power of a symplectic object}

Let $\C$ be a symmetric tensor category, Frobenius exact if ${\rm char}(\k)>0$. 
Let $V\in \C$ be a symplectic object (i.e., equipped with a non-degenerate skew-symmetric form) 
and assume that $SV$ has finite length. Following \cite{CEN}, let $m(V)$ be the largest integer $m$ for which $S^mV\ne 0$. 
By \cite[Proposition~2.3.2]{CEN}, the object $\psi=S^{m(V)}V$ is invertible.

\begin{lemma}\label{psisq} (i) For each $0\le i\le m(V)$ there are natural isomorphisms 
$$
S^iV\cong (S^iV)^*\cong S^{m(V)-i}V\otimes \psi
$$ 
and $\psi\otimes \psi\cong {\mathbf 1}$. 

(ii) $c_{\psi,\psi}=(-1)^{m(V)}\id_{\psi\otimes \psi}$. In particular, if $m(V)$ is odd, then $\psi$ is a super-line. 
\end{lemma} 
 
\begin{proof}
(i) Write $m=m(V)$ and $P=SV$. By Corollary~\ref{simplic2}(ii), multiplication by $V$ and contraction by $V^*$ give an isomorphism
$A(V\oplus V^*)\cong\underline{\End}(P)$. Writing multiplication operators to the left of contractions identifies the source, as an object, with $SV\otimes S(V^*)$. The reverse order gives $S(V^*)\otimes SV$.

Consider the block of an endomorphism from $S^iV$ to $S^0V=\mathbf1$. Projection onto this block is a split epimorphism. Order operators with multiplications on the left, so that contractions act first. To contribute to this block one must contract all $i$ factors and then multiply by a scalar. Thus the contraction pairing gives an epimorphism $S^i(V^*)\to(S^iV)^*$. Its source and target have the same length since $B$ identifies $V$ with $V^*$. It is therefore an isomorphism. In particular $S^iV$ is self-dual, and the invertible object $\psi=S^mV$ satisfies $\psi^{\otimes2}\cong\mathbf1$.

Next consider the block from $S^iV$ to $S^mV=\psi$. Order operators with multiplications on the right. Multiplication on $S^iV$ can raise degree by at most $m-i$, so only operators with no contractions contribute. We obtain an epimorphism
\[
 S^{m-i}V\longrightarrow (S^iV)^*\otimes\psi
\]
induced by multiplication. Interchanging $i$ and $m-i$ gives the reverse inequality of lengths, since $\psi$ is invertible. Both maps are isomorphisms. Thus multiplication followed by the projection $SV\to\psi$ defines a perfect $\psi$-valued pairing. Equivalently, $(SV)^*\cong SV\otimes\psi^*$ as $SV$-modules, where the action on the dual is transposed from multiplication. This proves (i).

(ii) In positive characteristic apply \cite[Lemma~5.3]{CFNO} to the finite graded Hopf algebra $SV$. This lemma identifies the scalar by which the antipode acts on the invertible top component with its self-braiding. The antipode of $SV$ acts by $-1$ on the primitive generators, hence by $(-1)^m$ on $\psi$. In characteristic zero, Lemma~\ref{lemo} and \cite{De} give a fiber functor on the subcategory generated by $V$ to $\sVec$. The image of $V$ is purely odd since its symmetric algebra is finite. Its top exterior power has parity $m$, giving the same formula.
\end{proof}

\subsection{Queer algebras in symmetric tensor categories}
Let $\C$ be a symmetric tensor category with a fixed super-line $\psi$. Recall that 
for $V\in \C$ the {\it queer algebra} $Q(V)$ is the image of the embedding 
$$
\underline{\End}(V)\oplus \underline\Hom(V,V\otimes \psi)\hookrightarrow \underline{\End}(V\oplus V\otimes \psi)
$$
 given by the schematic formula 
$(A,C)\mapsto \begin{pmatrix} A & C\\ C& A\end{pmatrix}$ (using that a morphism $C:V\to V\otimes\psi$ also gives a morphism $V\otimes\psi\to V$, which we also denote by $C$).

If $V\ne 0$, the algebras $\underline{\End}(V)$ and $Q(V)$ are simple, since they are 
Morita equivalent to simple algebras $\mathbf{1}$ and $Q(\mathbf{1})$. 
Namely, the $\C$-module category of $\underline{\End}(V)$-modules 
is equivalent to $\C$ via the assignment $X\in \C\mapsto X\otimes V$, while 
the $\C$-module category of $Q(V)$-modules is equivalent to 
$\C\boxtimes_\sVec\Vec$, where $\sVec$ is generated by $\psi$. 
The category $\C\boxtimes_\sVec\Vec$ is the category of objects 
$X\in\C$ equipped with an isomorphism $\gamma: X\to X\otimes \psi$ 
such that 
\[
(\id_X\otimes\alpha)\circ(\gamma\otimes\id_\psi)\circ\gamma=\id_X,
\]
which is the same as $Q(\mathbf{1})$-modules 
in $\C$, and the equivalence of categories $\C\boxtimes_\sVec\Vec\to Q(V)$-mod
is given by $X\mapsto X\otimes V$.
Under this equivalence $\psi\subset Q(\mathbf{1})=\mathbf{1}\oplus\psi$ acts by $\gamma$, while tensoring with $V$ implements the Morita equivalence from $Q(\mathbf{1})$ to $Q(V)$.

It is easy to see that $Q(V)\cong Q(V\otimes \psi)$ and 
$$
Q(V\otimes W)\cong Q(V)\otimes \underline{\End}(W)\cong \underline{\End}(V)\otimes Q(W),
$$
$$
Q(V)\otimes Q(W)\cong\underline{\End}(V\otimes W\oplus V\otimes W\otimes \psi).
$$

\subsection{Clifford and Weyl algebras in the Verlinde category}
Recall that $\Ver_p\cong \Ver_p^+\boxtimes \sVec$, where $\sVec$ is generated by $\psi=L_{p-1}$. 

\begin{proposition} The only simple algebras in $\Ver_p$ are of the form $\underline{\End} V$ and $Q(V)$, 
$0\ne V\in \Ver_p$. 
\end{proposition} 

\begin{proof} By \cite{EO,Os} the only indecomposable semisimple (=exact) module categories over $\Ver_p$ are $\Ver_p$ and $\Ver_p^+$.
Indeed, the ADET graphs with Coxeter number $p$ are $A_{p-1}$ and $T_{\frac{p-1}{2}}$, giving $\Ver_p$ and $\Ver_p^+$, respectively.

If $A$ is a simple algebra in $\Ver_p$ then $A$ is exact (\cite{CSZ}), so the category $\M:=\mathrm{mod}\text{-}A$ of right $A$-modules is an indecomposable exact (i.e., semisimple) module category over $\Ver_p$. Indecomposability follows because a decomposition of $\M$ would give a nontrivial product decomposition of the algebra object $A$. Thus $\M\cong \Ver_p$ or $\M\cong \Ver_p^+$. 

Under either equivalence, let $V$ denote the image of the regular right module $A_A$. By the standard reconstruction of an algebra from a module category and a generator (cf. \cite[Subsection~7.10]{EGNO}),
$A\cong \underline{\End}_{\M}(V)$
as algebra objects of $\Ver_p$. Indeed, with the usual composition, endomorphisms of $A_A$ act by left multiplication. For the regular left module the corresponding algebra is $A^{\rm op}$. In the regular module category this gives $A\cong\underline{\End}(V)=V\otimes V^*$. In the second case, $\Ver_p^+=\Ver_p^+\boxtimes\Vec$ is a module category over $\Ver_p^+\boxtimes\sVec$. The internal endomorphism algebra of $\k\in\Vec$ over $\sVec$ is $Q(\mathbf1)$, since the odd line acts as the identity on $\Vec$. Thus the internal endomorphism algebra of $V\boxtimes\k$ is $\underline{\End}(V)\otimes Q(\mathbf1)\cong Q(V)$.
\end{proof} 

\begin{proposition}\label{prop:spin}
Set $S_1=\mathbf1$. For $2\le k\le (p-1)/2$, let $S_k$ be the image of the spin representation under the fiber functor $\Ver_p(\operatorname{Spin}(2k-1))\to\Ver_p$. For $1\le k\le (p-1)/2$ there are algebra isomorphisms
\[
 A(L_{p-2k+1})\cong A_\psi(L_{2k-1})\cong Q(S_k).
\]
\end{proposition}

\begin{proof}
For $k=1$ this is the one-generator Clifford superalgebra $Q(\mathbf1)$, after normalizing the square of its odd generator. Assume $k\ge2$, and put $d=2k-1$. Let $U$ and $S$ be the vector and spin representations of $\operatorname{Spin}(d)$. The classical spin representation identifies $\Cl(U)_+$ with $\underline{\End}(S)$. The volume element is odd, central for ordinary multiplication, and invariant under $\operatorname{Spin}(d)$. Normalize its square to $1$. Multiplication by it identifies the odd component with the even component. Placing the odd component in parity $\psi$ therefore gives the queer algebra on the spinor object.

The representations $U$ and $S$ are tilting. Apply the additive symmetric monoidal functor from tilting modules to $\Ver_p$, obtained by semisimplification and the Verlinde fiber functor. It sends $U$ to $L_d$ by restriction to the principal $SL_2$, and $S$ to $S_k$. For $0\le r\le d+1\le p-1$, the antisymmetrizing idempotent is defined, so this functor preserves its image $\wedge^rU$. In particular, $\wedge^{d+1}L_d=0$, and all higher exterior powers vanish as well.

The normalized antisymmetrizations of Clifford words split the classical filtration with graded terms $\wedge^r U$, $0\le r\le d$. Therefore the image algebra has a filtration with these same exterior powers of $L_d$ as its graded terms. Its generators satisfy the Clifford relation, giving a filtered map from the Clifford algebra of $L_d$ to the image algebra. By Proposition~\ref{pbw}, the map on associated graded algebras is the identity on $\wedge L_d$, so the filtered map is an isomorphism. Transporting the even matrix algebra and the normalized odd volume element gives $A_\psi(L_d)\cong Q(S_k)$.
Finally, since $L_{p-1}=\psi$ and $L_{2k-1}\otimes\psi\cong L_{p-2k+1}$ in $\Ver_p$, Lemma \ref{leweyl1}(i) identifies $A_\psi(L_{2k-1})$ with the Weyl algebra $A(L_{p-2k+1})$.
\end{proof} 

\subsection{The symplectic Witt group} 

Let $\C$ be a symmetric tensor category, Frobenius exact if ${\rm char}(\k)>0$. Let $V\in \C$ be a symplectic object, and assume that $SV$ has finite length. In this case by Corollary \ref{simplic2}(iii), $A(V)$ is an Azumaya algebra in $\C$. 
Thus it defines a class $C_V$ in the Brauer group $\Br(\C)$. 

If $\C$ is finite, one can also describe this class by a module category. Since $A(V)$ is a simple, hence exact, algebra in $\C$, the category $\C_V:=A(V)$-mod is an indecomposable exact $\C$-module category. By the Azumaya property its equivalence class belongs to $\Pic(\C)$ (\cite{DN}, Subsection~5.5, and in the fusion case \cite{ENO}, Subsection~4.4). The canonical isomorphism $\Br(\C)\cong \Pic(\C)$ maps $C_V$ to $[\C_V]$.  

For any $\C$ as above, let ${\mathcal S\mathcal W}(\C)$ be the abelian group generated by the classes $[V]$ of symplectic objects $V$ in $\C$ with $SV$ of finite length under direct sum, with defining relation $[V]=0$ if $A(V)\cong \underline{\End}S$ for some $S\in \C$. This is an elementary abelian 2-group (an $\mathbb F_2$-vector space). We call ${\mathcal S\mathcal W}(\C)$ the {\it symplectic Witt group} of $\C$. 

Let $\Br_2(\C)$ be the 2-torsion subgroup of $\Br(\C)$. 

\begin{proposition}\label{symwitt0} The assignment $V\mapsto C_V$ defines a subgroup inclusion 
$$
\iota: {\mathcal S\mathcal W}(\C)\hookrightarrow \Br_2(\C).
$$ 
\end{proposition} 

\begin{proof}
By Lemma~\ref{leweyl1}(ii),
\[
 A(V\oplus W)\cong A(V)\otimes A(W),
\]
so $V\mapsto C_V$ takes orthogonal sums to products of Brauer classes. If $A(V)\cong\underline{\End}(S)$, its Brauer class is trivial, hence the defining relations of ${\mathcal S\mathcal W}(\C)$ lie in the kernel. Moreover, since $A(V)\cong A(V)^{\rm op}$, 
\[
 A(V)\otimes A(V)\cong A(V)\otimes A(V)^{\rm op}
 \cong\underline{\End}(A(V)),
\]
where the second isomorphism is the Azumaya property from Corollary~\ref{simplic2}(iii); thus $2[V]=0$ by the defining relation. This gives a homomorphism 
$$
\iota:{\mathcal S\mathcal W}(\C)\to\Br_2(\C).
$$

It remains to check injectivity. If a finite sum $[V_1]+\cdots+[V_r]$ maps to zero, set $V=V_1\oplus\cdots\oplus V_r$. Then $A(V)$ is Brauer-trivial. An equivalence $A(V)\text{-mod}\simeq\C$ sends the free rank-one module to an object $S\in\C$, and reconstruction from the regular left module gives $A(V)^{\rm op}\cong\underline{\End}(S)$. Taking opposites gives $A(V)\cong\underline{\End}(S^*)$. Hence $[V]=0$ is already one of the defining relations of ${\mathcal S\mathcal W}(\C)$. Therefore $\iota$ is injective.
\end{proof}

\subsection{\texorpdfstring{The symplectic Witt group of ${\rm Rep}(G)\boxtimes \sVec$ for a finite group $G$}{The symplectic Witt group of Rep(G) x sVec for a finite group G}}\label{sec:cliffgroup} 
Let $G$ be a finite group whose order is invertible in $\k$. Let $\C={\rm Rep}(G)\boxtimes \sVec$ and $\Gamma:=G\times \mathbb Z/2$ where $\mathbb Z/2$ is generated by an element $t$.
 Thus $\C={\rm Rep}(\Gamma,t)$, the super-Tannakian category of representations of $G$ on supervector spaces; the element $t$ is defined to act on such a representation by the parity operator. 
By Carnovale's computation \cite[Section~2]{Ca}, in the notation of \cite[Theorem~6.5]{DN}, 
$$
\Pic(\C)=H^2_*(\Gamma,\k^\times)\times \mu_2,
$$ 
with $H^2_*$ denoting $H^2$ with modified multiplication
 defined at the level of 2-cocycles 
by the formula
$$
(\beta*\gamma)(x,y)=(-1)^{\varepsilon_\beta(x)\varepsilon_\gamma(y)}\beta(x,y)\gamma(x,y),
$$
where $\varepsilon_\beta: \Gamma\to \mathbb Z/2$ is the homomorphism 
 such that
 $$
 (-1)^{\varepsilon_\beta(x)}=\beta(x,t)/\beta(t,x)\in \mu_2
$$
(it depends only on the cohomology class of $\beta$). 

If $X$ is a representation of $G$, denote by $X_-$ its extension to a representation of $\Gamma$ by letting $t$ act by $-1$. 
Also recall that ${\mathcal S\mathcal W}(\C)$ is generated by $[X\otimes \psi]$ 
where $X$ is an orthogonal representation of $G$. Let 
$\widetilde X:=X_-\oplus \k_-^{\dim X}$; thus the action of $\Gamma$ on $\widetilde X$ is given by 
$$
\widetilde\rho(g,t^a):=(-1)^a(\rho(g)\oplus \id_{\k^{\dim X}})\in O(\widetilde X).
$$ 
This homomorphism gives rise to a class $\beta_X\in H^2_*(G\times \mathbb Z/2,\k^\times)$ 
which is the pullback by $\widetilde\rho$ of the central extension 
$$
1\to \k^\times\to {\rm Cliff}(\widetilde X)\to O(\widetilde X)\to 1,
$$
where ${\rm Cliff}(\widetilde X)$ is the Clifford group of $\widetilde X$. 

\begin{proposition}\label{symwitt} The map $\iota$ is given by 
$$
\iota([X\otimes \psi])=(\beta_X,(-1)^{\dim X}),
$$ 
and it is a group homomorphism. Thus $\iota({\mathcal S\mathcal W}(\C))=\mathcal S\mathcal W_+(\C)\times \mu_2$, where $\mathcal S\mathcal W_+(\C)$ 
is the subgroup of $H^2_*(G\times \mathbb Z/2,\k^\times)$ generated by the $\beta_X$. 
\end{proposition} 

\begin{proof} Under the equivalence
\[
\operatorname{Rep}(G)\boxtimes\sVec\simeq \operatorname{Rep}(\Gamma,t),
\]
the algebra $A(X\otimes\psi)=A_\psi(X)$ is the usual Clifford superalgebra $\Cl(X_-)$. 
Write $n=\dim X$. The image of $[\Cl(X_-)\text{-mod}]$ under restriction to $\operatorname{Pic}(\sVec)=\mu_2$ is $(-1)^n$: after forgetting the $G$-action, $\Cl(X_-)$ is the graded tensor product of $n$ copies of the one-generator odd Clifford algebra, which represents the nontrivial element of $\operatorname{Pic}(\sVec)$.

We next compute the first coordinate. Choose homogeneous lifts $u_\gamma\in\operatorname{Cliff}({\widetilde X})$ of $\widetilde\rho(\gamma)$, $\gamma\in\Gamma$, and write
\[
u_\gamma u_\delta=b(\gamma,\delta)u_{\gamma\delta}.
\]
Then $[b]=\beta_X$. Let $p(\gamma)=|u_\gamma|\in \mathbb Z/2$, i.e., $\det \widetilde \rho(\gamma)=(-1)^{p(\gamma)}$; thus $p: \Gamma\to \mathbb Z/2$ is a group homomorphism. Choose a normalized volume element $\omega\in\Cl({\widetilde X})$ with $\omega^2=1$. It is even and satisfies
\[
\omega a=(-1)^{|a|}a\omega
\]
for every homogeneous $a\in\Cl({\widetilde X})$. Choose $i\in\k$ with $i^2=-1$ and define
\[
T_\gamma:=i^{p(\gamma)}\omega^{p(\gamma)}u_\gamma.
\]
The Clifford-group map is the twisted adjoint map, so for $v\in {\widetilde X}$,
\[
\widetilde\rho(\gamma)(v)=(-1)^{p(\gamma)}u_\gamma vu_\gamma^{-1}.
\]
It follows from the defining property of $\omega$ that
\[
T_\gamma vT_\gamma^{-1}=\widetilde\rho(\gamma)(v).
\]
Moreover, since $u_\gamma\omega^{p(\delta)}=(-1)^{p(\gamma)p(\delta)}\omega^{p(\delta)}u_\gamma$, $\omega^2=1$ and $i^2=-1$, we obtain
\[
T_\gamma T_\delta=b(\gamma,\delta)T_{\gamma\delta}.
\]
Now, since $\widetilde X$ is even-dimensional, $\Cl({\widetilde X})$ is a matrix superalgebra, therefore $\Cl({\widetilde X})\cong\underline{\operatorname{End}}(S)$ for a graded spinor module $S$. Since $\widetilde\rho(t)=-\id_{\widetilde X}$ and $p(t)=0$, we may normalize the lift $u_t$ to be the volume element $\omega$; then $T_t=\omega$ acts on $S$ as the parity operator. Thus the grading on $S$ is exactly the one determined by $t$, as required in the projective-representation description of the $H^2$-coordinate in the proof of \cite[Theorem~6.5]{DN}.
 Under this isomorphism the $\Gamma$-action on $\Cl({\widetilde X})$ is induced by the projective action $T_\gamma$ on $S$, whose multiplier is $b$. By the kernel description in the proof of \cite[Theorem~6.5]{DN}, it follows that
\[
[\Cl({\widetilde X})\text{-mod}]=(\beta_X,1).
\]

On the other hand, the orthogonal direct-sum formula for Clifford algebras gives a $\Gamma$-equivariant isomorphism
\[
\Cl({\widetilde X})\cong\Cl(X_-)\otimes\Cl(\k_-)^{\otimes n},
\]
where the tensor products are graded. The algebra $\Cl(\k_-)=\mathbf 1\oplus\k_-$ represents $(1,-1)$: its restriction to $\sVec$ is the nontrivial Brauer--Wall class, while its first coordinate is trivial because its $\Gamma$-action factors through $\langle t\rangle$ and $H^2(\langle t\rangle,\k^\times)$ vanishes. Consequently
\[
(\beta_X,1)=[\Cl({\widetilde X})\text{-mod}]
=[\Cl(X_-)\text{-mod}]\,(1,-1)^n,
\]
and hence\footnote{Since $\iota$ is a group homomorphism by Proposition~\ref{symwitt0},
this formula is compatible with direct sums.}
\[
\iota([X\otimes\psi])=[\Cl(X_-)\text{-mod}]
=(\beta_X,(-1)^n).
\]

Finally, for $X=\k$, the stabilized representation ${\widetilde X}=\k_-\oplus\k_-$ factors through $\langle t\rangle$, so $\beta_\k=1$ and
\[
\iota([\psi])=(1,-1).
\]
Thus the full factor $\mu_2$ lies in the image. For arbitrary $X$,
\[
(\beta_X,1)=\iota([X\otimes\psi])\,(1,-1)^n
\]
also lies in the image. Hence, identifying ${\mathcal S\mathcal W}(\C)$ with its image under $\iota$, that image is exactly $\mathcal S\mathcal W_+(\C)\times\mu_2$.
\end{proof}

\begin{remark}\label{precrel} The proof of Proposition \ref{symwitt} gives the precise relation
 with the pullback of the Clifford-group extension for the unstabilized representation $X_-$. Namely, the direct-product decomposition $\Gamma=G\times\langle t\rangle$ gives an isomorphism of ordinary cohomology groups
\[
\Theta:H^2(\Gamma,\k^\times)\xrightarrow{\ \sim\ }
H^2(G,\k^\times)\times\Hom(G,\mu_2)
\]
characterized by
\[
\Theta([c])=
\left([c|_{G\times G}],\ g\longmapsto\frac{c(t,g)}{c(g,t)}\right).
\]
Indeed, $H^2(\langle t\rangle,\k^\times)$ vanishes, and the second factor is the mixed term. Let $b_X^{\rm raw}$ be the pullback class along
\[
\Gamma\longrightarrow O(X),\qquad (g,t^a)\longmapsto(-1)^a\rho(g),
\]
and let $b_G(X)\in H^2(G,\k^\times)$ be the pullback class along $\rho$. More generally, if $U$ is an $m$-dimensional orthogonal $\Gamma$-representation on which $t$ acts by $-1$, a volume element $\omega_U$ is a Clifford lift of $t$, and for a homogeneous lift $u_g$ of $g\in G$ one has
\[
\omega_Uu_g=(-1)^{(m-1)|u_g|}u_g\omega_U,
\qquad
(-1)^{|u_g|}=\det(g|_U).
\]
Consequently
\[
\frac{\widetilde b_U(t,g)}{\widetilde b_U(g,t)}=\det(g|_U)^{m-1},
\]
where $\widetilde b_U$ is the cocycle obtained from the chosen Clifford lifts for $b_U^{\rm raw}$.
Applying this first to $U=X_-$ and then to $U=\widetilde X$, and observing that the stabilizing summand is $G$-trivial, gives
\[
\Theta(b_X^{\rm raw})=\bigl(b_G(X),{\rm det}_X^{n-1}\bigr),
\qquad
\Theta(\beta_X)=\bigl(b_G(X),{\rm det}_X\bigr).
\]
For $\chi\in\Hom(G,\mu_2)$ define the mixed cocycle
\[
\kappa_\chi\bigl((g,t^a),(h,t^b)\bigr):=\chi(h)^a.
\]
It has trivial restriction to $G\times G$ and satisfies
\[
\frac{\kappa_\chi(t,g)}{\kappa_\chi(g,t)}=\chi(g).
\]
Thus
\[
\beta_X=b_X^{\rm raw}\,\kappa_{\det_X}^{\dim X}
\quad\text{in }H^2(\Gamma,\k^\times).
\]
In particular, the raw pullback is already the first Picard coordinate when $n$ is even, whereas in odd dimension it must be corrected by $\kappa_{\det_X}$.

For two orthogonal $G$-representations $X,Y$, the isomorphism
\[
\Cl((X\oplus Y)_-)\cong\Cl(X_-)\otimes\Cl(Y_-)
\]
shows that these classes multiply. Equivalently,
\[
\beta_{X\oplus Y}=\beta_X*\beta_Y,
\]
and the second coordinates satisfy
\[
(-1)^{\dim(X\oplus Y)}=(-1)^{\dim X}(-1)^{\dim Y}.
\]
This gives a direct proof that the formula of Proposition \ref{symwitt} defines a group homomorphism.
\end{remark} 

\subsection{The Stiefel--Whitney subgroup}\label{sec:swgroup}

Let $\k=\mathbb C$. Put
\[
A:=H^2(G,\k^\times),\qquad K:=\Hom(G,\mathbb Z/2)=H^1(G,\mathbb Z/2),
\]
and let
\[
j:H^2(G,\mathbb Z/2)\longrightarrow H^2(G,\k^\times)
\]
be induced by the inclusion $\mathbb Z/2\hookrightarrow\k^\times$. Under the ordinary K\"unneth isomorphism
\[
\Theta:H^2(\Gamma,\k^\times)\xrightarrow{\sim}A\times K
\]
from Remark \ref{precrel}, the modified multiplication is
\begin{equation}\label{modifiedKunneth}
(a,\chi)*(b,\eta)=\bigl(abj(\chi\smile\eta),\chi+\eta\bigr).
\end{equation}
Indeed, the extra sign in the definition of $*$ restricts to $G\times G$ as the cocycle representing $j(\chi\smile\eta)$. 
Thus, we have a canonical short exact sequence of abelian groups 
$$
1\to H^2(G,\k^\times)\to H^2_*(G\times \mathbb Z/2,\k^\times)\to \Hom(G,\mathbb Z/2)\to 1.
$$
Notice that
\begin{equation}\label{cupsquarezero}
j(\chi\smile\chi)=1
\end{equation}
for every $\chi\in K$; indeed, the cocycle $(-1)^{\chi(g)\chi(h)}$ is the coboundary of the cochain $g\mapsto i^{\chi(g)}$, where $i^2=-1$.

For an orthogonal representation $X$ of $G$, let $w_1(X)\in K$ and $w_2(X)\in H^2(G,\mathbb Z/2)$ be its first and second Stiefel--Whitney classes (i.e., of the real rank-$\dim_{\mathbb C}X$ bundle on $BG$ obtained by reducing $O(\dim X,\mathbb C)$ to its maximal compact subgroup); thus $\det_X=(-1)^{w_1(X)}$. The obstruction to lifting orthogonal transition functions to $\operatorname{Pin}^+$ is $w_2(X)$. Pushing the double cover out along $\mu_2\hookrightarrow\mathbb C^\times$ gives the Clifford-group extension. With the other Pin convention the obstruction changes by $w_1(X)^2$, whose image under $j$ vanishes by \eqref{cupsquarezero}. Hence
\[
b_G(X)=j(w_2(X))\in A[2],
\qquad
\Theta(\beta_X)=\bigl(b_G(X),w_1(X)\bigr).
\]
Define
\[
{\mathcal S\mathcal W}(G):=
\left\langle b_G(X)\ \middle|\ X\text{ is an orthogonal representation of }G\right\rangle
\subset A[2].
\]
We call this the \emph{Stiefel--Whitney subgroup} of $G$.

\begin{proposition}\label{SWproductstructure} (Symplectic Witt via Stiefel-Whitney)
Under $\Theta$ one has
\[
\mathcal S\mathcal W_+(\C)={\mathcal S\mathcal W}(G)\times K
\]
as a set. In particular, $\mathcal S\mathcal W_+(\C)$ is also a subgroup of $H^2(\Gamma,\k^\times)$ for the ordinary multiplication.

Moreover, let
\[
D(G):=
\left\langle j(\chi\smile\eta)\ \middle|\ \chi,\eta\in K\right\rangle
\subset A.
\]
Then $D(G)\subset {\mathcal S\mathcal W}(G)$, and
\[
L:=D(G)\times K\subset \mathcal S\mathcal W_+(\C)
\]
is a subgroup for the modified multiplication. Namely, it is the subgroup generated by the elements $(1,\chi)$, $\chi\in K$. 
\end{proposition}

\begin{proof} The Whitney sum formula
$$
w_2(X\oplus Y)=w_2(X)+w_2(Y)+w_1(X)\smile w_1(Y)
$$
yields
\begin{equation}\label{Whitneysumdegree2}
b_G(X\oplus Y)=b_G(X)b_G(Y)j\bigl(w_1(X)\smile w_1(Y)\bigr).
\end{equation}
Since every element of $K$ is $w_1$ of an orthogonal line, it follows that $D(G)\subset {\mathcal S\mathcal W}(G)$. Formula \eqref{modifiedKunneth} and Proposition~\ref{symwitt} therefore imply
\[
\mathcal S\mathcal W_+(\C)\subset {\mathcal S\mathcal W}(G)\times K.
\]
Conversely, for every $\chi\in K$, the corresponding orthogonal line $\k_\chi$ satisfies
\[
\Theta(\beta_{\k_\chi})=(1,\chi).
\]
Hence, by \eqref{modifiedKunneth} and \eqref{cupsquarezero},
\[
\Theta(\beta_X)*(1,w_1(X))=(b_G(X),0).
\]
Thus $\mathcal S\mathcal W_+(\C)$ contains ${\mathcal S\mathcal W}(G)\times\{0\}$ and $\{1\}\times K$, and hence contains their Cartesian product. This proves the asserted equality of sets, and the statement about the ordinary multiplication follows immediately.

The subset $L$ is closed under $*$ by \eqref{modifiedKunneth}, and every element of $L$ is its own inverse by \eqref{cupsquarezero}; hence it is a subgroup. Also,
\[
(1,\chi)*(1,\eta)*(1,\chi+\eta)=\bigl(j(\chi\smile\eta),0\bigr),
\]
so $L$ is generated by the elements $(1,\chi)$. 
\end{proof}

\begin{remark} Note that the subset $\{(1,\chi)\mid\chi\in K\}$ need not be a subgroup for $*$. Indeed, 
\[
(1,\chi)*(1,\eta)=\bigl(j(\chi\smile\eta),\chi+\eta\bigr),
\]
whose first coordinate may be nontrivial. For example, for $G=(\mathbb Z/2)^2$ and the two coordinate characters $\chi,\eta$, the class $j(\chi\smile\eta)$ is nonzero, as is seen from the nontrivial commutator of its representing cocycle.
\end{remark}

\begin{corollary}\label{SWsurjectivitycriterion}
(i) One has, as sets,
\[
H^2_*(\Gamma,\k^\times)[2]=H^2(G,\k^\times)[2]\times K
\]
and
\[
\iota\bigl({\mathcal S\mathcal W}(\C)\bigr)
={\mathcal S\mathcal W}(G)\times K\times\mu_2.
\]
Consequently,
\[
\iota:{\mathcal S\mathcal W}(\C)\longrightarrow\Br_2(\C)
\]
is surjective if and only if
\[
{\mathcal S\mathcal W}(G)=H^2(G,\k^\times)[2].
\]

(ii) $\iota$ is surjective if $G$ is abelian. 

(iii) $\iota$ is not always surjective.
\end{corollary}

\begin{proof}
(i) By \eqref{modifiedKunneth} and \eqref{cupsquarezero},
\[
(a,\chi)*(a,\chi)=(a^2,0),
\]
which proves the first equality. The second follows from Proposition~\ref{symwitt} and Proposition~\ref{SWproductstructure}, and gives the desired criterion.

(ii) For finite abelian $G$, the commutator bicharacter identifies $H^2(G,\mathbb C^\times)$ with the group of alternating bicharacters on $G$, since $\mathbb C^\times$ is divisible. A class of order dividing two takes commutator values in $\mu_2$ and so factors through $G/2G$ in both variables. Under this identification $j(\chi\smile\eta)$ corresponds to $\chi\wedge\eta$. Such decomposable forms span $\wedge^2 K$, so $D(G)=A[2]$ and (i) gives surjectivity.

(iii) The type-$1$ split metacyclic groups occurring in the proof of \cite[Theorem~2, pp.~336--337]{GKT} have a degree-two class outside the subgroup generated by second Stiefel--Whitney classes. Its image under
\[
j:H^2(G,\mathbb Z/2)\longrightarrow H^2(G,\k^\times)[2]
\]
lies outside ${\mathcal S\mathcal W}(G)$. Indeed, if $j(x)$ belonged to ${\mathcal S\mathcal W}(G)$, then $x$ would differ from a sum of second Stiefel--Whitney classes by an element of $\ker j$; but $\ker j$ consists of Kummer classes of complex characters, which are themselves second Stiefel--Whitney classes of the underlying oriented real two-planes. Thus ${\mathcal S\mathcal W}(G)\subsetneq H^2(G,\mathbb C^\times)[2]$, and the result follows from (i).
\end{proof}

\section[The Weyl algebra without Frobenius exactness]{The Weyl algebra without Frobenius exactness}\label{sec:reduced}

Let $(V,B)$ be a symplectic object in a symmetric tensor category $\C$.
In the proof of Proposition~\ref{simplic}(i), the role of the bound on
symmetric powers was to ensure that the first coproduct component
\[
 \delta_d=\Delta_{1,d-1}:S^dV\longrightarrow V\otimes S^{d-1}V
\]
is injective for $d>0$. We first show that this condition is also necessary.

\begin{theorem}\label{thm:criterion}
The Weyl algebra $A(V)$ is simple if and only if $\delta_d$ is a
monomorphism for every $d>0$. If $SV$ has finite length, these conditions
are also equivalent to $A(V)$ being Azumaya.
\end{theorem}

\begin{proof}
Use the Moyal--Weyl identification $A(V)=(SV,\star)$ from the proof of
Proposition~\ref{pbw}. Write $M_v$ for ordinary multiplication by $v$
and $D_v$ for contraction with $B(v,-)$. In schematic notation
left and right multiplication by a generator, with the symmetry
understood in right multiplication, are respectively
\begin{equation}\label{eq:generatoractions}
 L_v=M_v+\tfrac12D_v,\qquad R_v=M_v-\tfrac12D_v.
\end{equation}
Indeed, a factor of degree one allows at most one contraction.
If $K=\ker\delta_d\ne0$ for some $d>0$, every first contraction
vanishes on $K$. By the Leibniz rule the ordinary homogeneous ideal
$(SV)K$ is stable under all contractions. By \eqref{eq:generatoractions}
it is a two-sided ideal for $\star$, since $V$ generates $A(V)$.
It contains $K$ and has zero component in degree zero, so it is proper.
Conversely, if all $\delta_d$ are monic, the least-filtration-degree
argument in Proposition~\ref{simplic}(i) applies verbatim. In fact it
proves that $R\otimes A(V)$ is simple for every simple algebra $R$:
tensoring $\delta_d$ with $R$ preserves monomorphisms, so a nonzero
leading term of positive degree has a nonzero contraction.

Suppose now that $SV$ has finite length and all $\delta_d$ are monic.
Their iterates give monomorphisms from $S^dV$ to $V^{\otimes d}$.
Their images lie in the symmetric invariants
$(V^{\otimes d})^{\mathfrak S_d}=(S^d(V^*))^*$. The latter has
the same length as $S^dV$, since $B$ identifies $V$ with $V^*$.
Thus $SV$ and hence $A(V)$ are self-dual as objects.
Taking $R=A(V)^{\rm op}$ in the preceding paragraph shows that the
source of the unital left-right action
$A(V)\otimes A(V)^{\rm op}\to\underline{\End}(A(V))$ is simple.
The action is therefore monic, and self-duality makes its source and
target isomorphic as objects of finite length. Hence it is an isomorphism.
Conversely, every nonzero Azumaya algebra $A$ is simple: a two-sided
ideal is stable under $\underline{\End}(A)=A\otimes A^*$, and for any
nonzero subobject $I\subset A$ the restricted evaluation
$A^*\otimes I\to\mathbf1$ is nonzero by rigidity, hence epic.
Consequently $\underline{\End}(A)I=A$.
\end{proof}

\subsection{A counterexample in every higher Verlinde category}
\begin{proposition}\label{prop:higherfailure}
Let $p={\rm char}(\k)>2$ and $n\ge2$. Let $U$ be the image of the natural two-dimensional
$SL_2$-module in $\Ver_{p^n}$, with its determinant form.
Then $SU$ has top degree $p^n-2$, but $\ker\delta_p\ne0$.
In particular $A(U)$ has a nonzero proper two-sided ideal and is
not Azumaya.
\end{proposition}

\begin{proof}
The assertion about the top degree is \cite[Corollary~6.4]{BE}.
By \cite[Proposition~6.1]{BE}, the degree-$p$ Koszul sequence is exact:
\begin{equation}\label{eq:higherkoszul}
 0\longrightarrow S^{p-2}U\longrightarrow U\otimes S^{p-1}U
 \xrightarrow{\ m\ } S^pU\longrightarrow 0 .
\end{equation}
The composite $m\delta_p=p\,\id_{S^pU}$ vanishes, so $\delta_p$
factors through $S^{p-2}U$. For $\theta=\pi/p^n$,
\cite[Corollary~6.3]{BE} gives
\[
 \operatorname{FPdim}(S^jU)=\frac{\sin((j+1)\theta)}{\sin\theta},
 \qquad 0\le j\le p^n-2.
\]
Hence
\[
 \operatorname{FPdim}(S^pU)-\operatorname{FPdim}(S^{p-2}U)
 =2\cos(p\theta)>0 .
\]
Thus $\delta_p$ has nonzero kernel. The ideal is the ordinary ideal generated by
this kernel in the Moyal--Weyl realization of Theorem~\ref{thm:criterion}.
Taking $W=U$, the hyperbolic conclusion of Corollary~\ref{simplic2}(ii) likewise fails without Frobenius exactness: $K=\ker\delta_p$ embeds in $S^p(W\oplus W^*)$ through the first
summand and is killed by all first contractions there. Thus
$A(W\oplus W^*)$ has a proper ideal, whereas $\underline{\End}(SW)$ is simple.
\end{proof}

\subsection{An Azumaya quotient}
For any object $X$, let $N_d(X)$ denote the norm map
\begin{equation}\label{eq:nicholsnorm}
 S^dX=(X^{\otimes d})_{\mathfrak S_d}\longrightarrow
 (X^{\otimes d})^{\mathfrak S_d}=(S^d(X^*))^*,
 \qquad [t]\longmapsto\sum_{\sigma\in\mathfrak S_d}\sigma t .
\end{equation}
Let $H_d(X)=\operatorname{im}N_d(X)$, with $N_0=\id$. Set
\[
 H(X)=\bigoplus_d H_d(X),\qquad J(X)=\bigoplus_d\ker N_d(X).
\]
The graded algebra $H(X)$ is the Nichols algebra for the given symmetry.
Indeed, \eqref{eq:nicholsnorm} is the graded Hopf pairing of $SX$
with $S(X^*)$ extending evaluation in degree one. On products of
generators it sums over all matchings. Separating the matches to the
two factors of a product gives the Hopf pairing identity. Thus $J(X)$
is a homogeneous ideal, and transposition gives perfect pairings
between $H_d(X)$ and $H_d(X^*)$.

\begin{theorem}\label{thm:reduced}
In the Moyal--Weyl realization of $A(V)$, the subobject $J(V)$ is
a two-sided ideal. The quotient $A_{\rm red}(V):=A(V)/J(V)$ has
associated graded algebra $H(V)$ and is simple. If $H(V)$ has finite
length, $A_{\rm red}(V)$ is Azumaya. In odd positive characteristic
this finiteness holds for every $V$ in every symmetric tensor category
of moderate growth.
\end{theorem}

\begin{proof}
The iterated coproduct to $d$ factors of degree one is the norm
$N_d$ followed by the inclusion into $V^{\otimes d}$. Thus
\begin{equation}\label{eq:normrecursion}
 \iota_d N_d=(\id_V\otimes \iota_{d-1}N_{d-1})\delta_d
\end{equation}
where $\iota_j:(V^{\otimes j})^{\mathfrak S_j}\hookrightarrow V^{\otimes j}$ is the inclusion. Exactness of tensor product gives
$\delta_d(J_d)\subset V\otimes J_{d-1}$, where $J_d=\ker N_d$.
Hence $J(V)$ is stable under contractions, and
\eqref{eq:generatoractions} proves it is a Weyl ideal.
It is homogeneous as a subobject of $SV$, so PBW identifies the
associated graded of the quotient with $H(V)$, and its underlying
object with $\bigoplus_d H_d(V)$.

By \eqref{eq:normrecursion} the induced map
$H_d(V)\to V\otimes H_{d-1}(V)$ is monic: its composite with
the inclusion into $V^{\otimes d}$ is the inclusion of $H_d(V)$.
The proof of Theorem~\ref{thm:criterion} by least filtration degree
therefore gives simplicity, also after tensoring with any simple algebra
$R$. When $H(V)$ has finite length it is self-dual by the perfect
pairings above and $V\cong V^*$. Applying the same left-right action
argument to $A_{\rm red}(V)$ proves the Azumaya assertion.

For the last assertion place $V$ in odd degree in
$\C\boxtimes\sVec$, which still has moderate growth. The antisymmetrizer
on its $d$-th tensor power acts as the ordinary symmetrizer on
$V^{\otimes d}$ tensored with $\psi^{\otimes d}$. By
\cite[Proposition~5.1]{CFNO}, its image vanishes for all sufficiently
large $d$. Thus only finitely many $H_d(V)$ are nonzero.
\end{proof}

If all $\delta_d$ are monic, \eqref{eq:normrecursion} gives $J(V)=0$
by induction. Conversely, a nonzero kernel of $\delta_d$ lies in $J(V)$.
Thus the reduction recovers the original Weyl algebra precisely in
the simple case. For an ordinary symplectic vector space in
characteristic $p>2$, a monomial calculation gives the familiar
quotient by the central $p$-th powers: the norm of a basis monomial
with exponents $a_i$ is $\prod_i a_i!$ times the sum of its distinct
orderings.

\begin{corollary}\label{cor:reducedhyperbolic}
For a hyperbolic object $W\oplus W^*$ with $H(W)$ of finite length,
one has $A_{\rm red}(W\oplus W^*)\cong\underline{\End}(H(W))$.
For an orthogonal direct sum of symplectic objects,
\[
 A_{\rm red}(V\oplus V')\cong
 A_{\rm red}(V)\otimes A_{\rm red}(V').
\]
In particular, in odd positive characteristic and moderate growth
the assignment $V\mapsto[A_{\rm red}(V)]$ takes orthogonal sums
to products in $\Br_2(\C)$ and vanishes on hyperbolic objects.
\end{corollary}

\begin{proof}
For any $X,Y$ the degreewise pairing in \eqref{eq:nicholsnorm}
pairs only equal bidegrees in $S(X\oplus Y)=SX\otimes SY$.
Within each bidegree it is the tensor product of the two pairings,
since each matching must respect the two summands. Hence
$H(X\oplus Y)=H(X)\otimes H(Y)$. For an orthogonal sum the Moyal
product also factors, since there are no cross-contractions.
This gives the asserted algebra isomorphism.

In the hyperbolic case both summands are isotropic, so $H(W)$ and
$H(W^*)$ are subalgebras of $A_{\rm red}(W\oplus W^*)$.
Their ordered multiplication is an isomorphism of filtered objects,
as its associated graded is the decomposition just proved. Induce the
augmentation module $\mathbf1$ from $H(W^*)$ to $A_{\rm red}(W\oplus W^*)$.
The resulting left module has underlying object $H(W)$, since ordered
multiplication is right $H(W^*)$-linear. Its action map is monic by
simplicity. The domain and codomain have the same finite length by
$H(W^*)\cong H(W)^*$, so the action is an isomorphism.
Finally, changing $B$ to $-B$ gives the opposite product and preserves
$J(V)$. Rescaling $V$ by a square root of $-1$ identifies the two
forms. Thus the reduced Azumaya class has order dividing two.
\end{proof}

For completeness, the top-power assertion has a reduced version as well.
If $V\ne0$ and $m$ is the top degree of $H(V)$ in odd positive
characteristic and moderate growth, $Z=H_m(V)$ is invertible and
$c_{Z,Z}=(-1)^m\id$. Indeed, Proposition~5.1 of \cite{CFNO}
applied to $V$ in odd degree gives the even invertible object
$Z\boxtimes\psi^{\otimes m}$. Since $V\cong V^*$, the perfect
degree-$m$ norm pairing gives $Z\cong Z^*$ and hence
$Z^{\otimes2}\cong\mathbf1$.

\subsection{The reduced algebra of the defining object}
We now compute the reduced Weyl algebra of the symplectic object $U$
in Proposition~\ref{prop:higherfailure}.

\begin{proposition}\label{prop:reduceddefining}
Let $p>2$, $n\ge2$ and $a=\frac{p-1}{2}$. Then
\begin{equation}\label{eq:reduceddefining}
 A_{\rm red}(U)\cong
 \underline{\End}\bigl(S^aU\oplus S^{a-1}U\bigr).
\end{equation}
Moreover, $H_d(U)\cong S^{\min(d,2p-2-d)}U$ for
$0\le d\le2p-2$, and $H_d(U)=0$ for $d>2p-2$.
\end{proposition}

\begin{proof}
First work in $\operatorname{Rep}(SL_2)$ and let $E=\k x\oplus\k y$
with $B(x,y)=1$. Its Nichols algebra is
$H(E)=\k[x,y]/(x^p,y^p)$. For $d<p$ we have $H_d(E)=S^dE$.
Multiplication into the top line spanned by $x^{p-1}y^{p-1}$
pairs complementary monomials perfectly. This line is trivial as an
$SL_2$-module. Thus in degrees $p\le d\le2p-2$,
$H_d(E)\cong(S^{2p-2-d}E)^*\cong S^{2p-2-d}E$. In particular,
all the nonzero $H_d(E)$ are simple tilting modules.

Put $P=S^aE\oplus S^{a-1}E$. For $v\in E$, define $T_v$ on $P$ by
$2D_v:S^aE\to S^{a-1}E$ and $M_v:S^{a-1}E\to S^aE$, where
$D_v$ is contraction with $B(v,-)$. This construction is
$SL_2$-equivariant. Since $D_x=\partial_y$ and $D_y=-\partial_x$,
the Euler identity gives
\[
 [T_x,T_y]|_{S^aE}=-2a\,\id,\qquad
 [T_x,T_y]|_{S^{a-1}E}=2(a+1)\,\id .
\]
Both scalars are $1$ in $\k$. Also $D_vM_v=M_vD_v$ because
$B(v,v)=0$. Moving all derivatives to the right in $T_v^p$ gives $a+1$ derivatives on $S^aE$ and $a$ on $S^{a-1}E$. Each number exceeds the degree of the input, so $T_v^p=0$. Thus the operators define an action
of the ordinary Weyl algebra with central $p$-th powers set to zero,
namely $A_{\rm red}(E)$. This algebra is simple even after forgetting
the $SL_2$-action, by Theorem~\ref{thm:reduced} in $\Vec$.
Since $\dim P=p$ and $\dim H(E)=p^2$, the action gives an equivariant
algebra isomorphism $A_{\rm red}(E)\cong\underline{\End}(P)$.

It remains to compare reduction with the additive tensor functor
$F:\operatorname{Tilt}(SL_2)\to\Ver_{p^n}$. In degree $d\le2p-2$,
factor the symmetrizer on $E^{\otimes d}$ as the projection onto
its image $H_d(E)$ followed by inclusion. Both are morphisms of
tilting modules. The functor is faithful on tiltings of highest
weights less than $p^n-1$, and $F(S^jE)=S^jU$ is simple for
$0\le j<p$ \cite[Proposition~3.5 and Lemma~4.29]{BEO}.
Since $2p-2<p^n-1$, both factor maps remain nonzero. Their middle
object is simple, so they remain an epimorphism and a monomorphism,
respectively. For $d>2p-2$ the symmetrizer on $E^{\otimes d}$ is zero.
This proves the formula for $H_d(U)$.

Write $q_d:E^{\otimes d}\twoheadrightarrow H_d(E)$ for the first factor map. The map $F(q_d)$ factors through $S^dU$, and the induced map has kernel $\ker N_d(U)$. This identifies $F(H(E))$ with $H(U)$ degree by degree.

To check the product under this identification, put $t=t_1\otimes\cdots\otimes t_d$ and write $t_{\widehat j}$ for the tensor with $t_j$ omitted. In the classical reduced algebra, formula~\eqref{eq:generatoractions} gives
\[
 L_v q_d(t)=q_{d+1}(v\otimes t)
       +\tfrac12\sum_{j=1}^d B(v,t_j)q_{d-1}(t_{\widehat j}).
\]
Both sides are morphisms from $E\otimes E^{\otimes d}$ to $H_{d+1}(E)\oplus H_{d-1}(E)$, with $H_{-1}=0$. The right side uses only $q_r$, the form and permutations of tensor factors. Its image under $F$ is therefore the same generator action on $H(U)$. Since each $F(q_d)$ is epic and $U$ generates the algebra, the degreewise identification preserves multiplication. Applying $F$ to the equivariant isomorphism $A_{\rm red}(E)\cong\underline{\End}(P)$ proves \eqref{eq:reduceddefining}.
\end{proof}

\end{document}